\documentclass[sn-mathphys-num]{sn-jnl}

\usepackage{graphicx}%
\usepackage{multirow}%
\usepackage{amsmath,amssymb,amsfonts}%
\usepackage{amsthm}%
\usepackage{mathrsfs}%
\usepackage[title]{appendix}%
\usepackage{xcolor}%
\usepackage{textcomp}%
\usepackage{manyfoot}%
\usepackage{booktabs}%
\usepackage{algorithm}%
\usepackage{algorithmicx}%
\usepackage{algpseudocode}%
\usepackage{listings}%

\newtheorem{theorem}{Theorem}[section]

\newtheorem{corollary}[theorem]{Corollary}

\newtheorem{definition}[theorem]{Definition}

\newtheorem{proposition}[theorem]{Proposition}

\begin{document}

\title[Fundamental Theorem of Algebra]{An Algebraic Approach to the Fundamental Theorem of Algebra}


\author*[1]{\fnm{Priyabrata} \sur{Mandal}}\email{p.mandal@manipal.edu}

\author[2]{\fnm{Sajad A.} \sur{Sheikh}}\email{sajadsheikh@uok.edu.in}


\affil*[1]{\orgdiv{Department of Mathematics}, \orgname{Manipal Institute of Technology, Manipal Academy of Higher Education}, \orgaddress{ \city{Manipal}, \postcode{576104}, \state{Karnataka}, \country{India}}}

\medskip 

\affil[2]{\orgdiv{Department of Mathematics}, \orgname{South Campus, University of Kashmir}, \orgaddress{ \city{Anantnag}, \postcode{192101}, \state{Jammu and Kashmir}, \country{India}}}

\abstract{
	In this paper, we investigate the algebraic counterpart of the Fundamental Theorem of Algebra. We explore the concept of {\it real-closed} fields and quadratic forms.  We show, by means of Galois theory, that $\mathcal F(\sqrt{-1})$ is algebraically closed if $\mathcal F$ is real-closed. Lastly, we explain the algebraic closure of $\mathbb R(\sqrt{-1})=\mathbb C$ by demonstrating the real-closeness of $\mathbb R$.
}




\maketitle

\section{Introduction}
The study of polynomials and their roots has historically occupied the interest of mathematicians, and the Fundamental Theorem of Algebra is the fruit of centuries of mathematical efforts in that direction. The Fundamental Theorem of Algebra states that every non-constant polynomial equation with complex coefficients has at least one complex root. This theorem has a rich history and has been proven using various techniques from different branches of mathematics, including analytic, algebraic, and topological methods. In the $16^{th}$ century, del Ferro and Tartaglia discovered solutions for cubic equations via radicals, and Cardano's student Ferrari extended this to quartic equations. However, despite considerable efforts over subsequent centuries, no such formula could be found for polynomials of degrees 5 and higher. Ruffini, Abel and Galois ultimately proved the non-existence of algebraic solutions for general polynomials beyond the quartic degree, resolving a long-standing problem in algebra \cite{boyer2012history, edwards1984galois}.

\medskip
  
  In the $17^{th}$ century, René Descartes and Pierre de Fermat also made significant contributions in advancing our  understanding of polynomial equations and their roots. However, it was not until the early $19^{th}$ century that the first rigorous proofs of the Fundamental Theorem of Algebra emerged, primarily through the work of Carl Friedrich Gauss \cite{gaussproof, stillwell2010mathematics}. Since Gauss's seminal work, numerous alternative proofs have emerged, offering fresh insights and employing a wide array of methods and approaches. Remarkably, nearly a hundred proofs of the Fundamental Theorem of Algebra were published by the year 1907 \cite{Remmert1995}, reflecting the theorem's significance and the enduring fascination of the mathematical community with it. The proliferation of proofs continued in the subsequent decades, with at least another hundred proofs published thereafter.
  The abundance of proofs for the Fundamental Theorem of Algebra parallels the multitude of proofs that exist for the celebrated Pythagorean theorem \cite{Maor2007}, highlighting the fundamental nature of these mathematical statements and their ability to inspire diverse and creative proofs across generations of mathematicians.
\medskip

  In addition, various efforts have focused on minimal algebraic and analytic requirements for proving the Fundamental Theorem of Algebra. Shipman \cite{shipman} demonstrated that any field where polynomials of prime degree have roots is algebraically closed, a result further simplified in \cite{aliabadi}. For more details on minimal analytic approaches, refer to \cite{basu}.

\medskip 

We begin by defining real closed fields and establishing their fundamental properties. We then introduce the concept of field extensions, leading to the construction of the field $\mathcal F(\sqrt{-1})$, where $\mathcal F$ is a real closed field. Using basic Galois theory, we demonstrate that $\mathcal F(\sqrt{-1})$ forms an algebraically closed field, laying the foundation for the main result of our proof.
Finally, we prove that the field of real numbers $\mathbb R$ is real closed, a key step in establishing the algebraic closure of the field of complex numbers $\mathbb C$. By showing that $\mathbb R(\sqrt{-1})=\mathbb C$ is algebraically closed, we establish the Fundamental Theorem of Algebra within an algebraic framework, expressing the deep connections between real and complex analysis.

\medskip

\section{Preliminaries}\label{prelim}
In this section, we introduce the notations that will be used throughout this paper. For more detailed explanations, readers are directed to the relevant sections in \cite{lam}. It is assumed that all fields discussed have characteristics different from $2$. 
\medskip 

For a vector space $V$ with a bilinear form $b$, two vectors $v_1$ and $v_2$ are said to be {\it orthogonal} to each other if $b(v_1,v_2)=0$. Let $V^\perp$ denote the set of all vectors orthogonal to every vector in $V$. The bilinear space $(V,b)$ is {\it non-singular} if $V^\perp=0$ (see \cite{sch}, Chapter $1$ for more details).
Consider a quadratic form $q$ over a field $\mathcal F$, with $V$ as its associated vector space. We call $(V,q)$ a { quadratic space} over $\mathcal F$. The dimension of $q$ is equal to the dimension of $V$ when considered as a vector space over $\mathcal F$, and is denoted by ${\rm dim}\ q$. An $n$-dimensional quadratic form is equivalent to $a_1v_1^2+a_2v_2^2+ \cdots +a_nv_n^2$ over $\mathcal F$. We denote the quadratic form $a_1v_1^2+a_2v_2^2+ \cdots +a_nv_n^2$ by $\langle a_1,\ldots,a_n\rangle$.
\medskip 

For a quadratic form $q$, one can associate a bilinear form $b_q$ defined by \[b_q(x,y):=\frac{1}{2}[q(x+y)-q(x)-q(y)]\] 
If the bilinear form $b_q$ is non-singular, then we say $q$ is a { non-singular} quadratic form. Throughout this article, a quadratic form always refers to a non-singular quadratic form.

\medskip 

Suppose $\sigma(\mathcal F)$ is the collection of members in $\mathcal F$ which are sum of squares in $\mathcal F$. We define $\mathcal F$ to be {\it formally real} if $-1$ cannot be written as a sum of squares in $\mathcal F$. In other words, $\mathcal F$ is formally real if $-1 \notin \sigma(\mathcal F)$. By a {\it Pythagorean field} we mean a field in which the sum of two squares resulting a square. Thus, for a Pythagorean field, $\sigma(\mathcal F)$ is the same as $\mathcal F^2$.
\medskip 

\begin{definition} (\cite{lam}, Chapter $8$, Definition $1.2$)
Let $\Sigma$ be a proper subset of $\mathcal F$. Then $\Sigma$ is said to be an ordering on $\mathcal F$ if 
\begin{enumerate}
    \item $\mathcal F^2\subseteq \Sigma$
    \item $\Sigma$ is closed under addition.
    \item $\Sigma$ is closed under multiplication.
    \item $\Sigma \cup -\Sigma=\mathcal F$
    \item $\Sigma \cap -\Sigma=\{0\}$.
\end{enumerate}
\end{definition}

Note that $-1 \notin \Sigma$. Indeed, if $-1 \in \Sigma$, then any $x \in F$ can be written as \[x=\left[\left(\frac{x+1}{2}\right)^2+(-1)\left(\frac{x-1}{2}\right)^2 \right] \in \Sigma+(\Sigma\cdot \Sigma) \subseteq \Sigma+\Sigma \subseteq \Sigma\]
Hence $\mathcal F \subseteq \Sigma$ and so $\Sigma=\mathcal F$. This contradicts the definition of an ordering.
Thus $-1 \notin \Sigma$. Also, since $\mathcal F^2 \subset \Sigma$ and $\Sigma+\Sigma \subset \Sigma$, we have $\sigma(\mathcal F) \subset \Sigma$. Combining the facts $\sigma(\mathcal F) \subset \Sigma$ and  $-1 \notin \Sigma$, we conclude $-1 \notin \sigma(\mathcal F)$. Therefore, $\mathcal F$ is formally real if and only if it has at least one ordering.


\section{Formally real fields}
In this section, we discuss formally real fields. For more details, we refer to (\cite{lam}, Chapter $8$). Recall that a field $\mathcal F$ is said to be formally real if $-1 \notin \sigma(\mathcal F)$. 
 Let $\mathcal F^\times$ denote the set of non-zero elements of $\mathcal F$, i.e., $\mathcal F^\times= \mathcal F \setminus \{0\}$. 
    A Euclidean field $\mathbb E$ is formally real with $\mathbb E^{\times}= \mathbb E^{\times 2} \cup (-\mathbb E^{\times2})$. 
    More precisely, a Euclidean field is an ordered field in which every non-negative element is a square. A field $\mathcal F$ is said to be {\it quadratically closed} if $\mathcal F^2=\mathcal F$.
\medskip 

\begin{proposition}[\cite{lam}, Chapter 8, Proposition 1.6]\label{Euclidean is Pythagorean}
    A field $\mathbb E$ is Euclidean if and only if $\mathbb E$ is a Pythagorean field with a unique ordering.
\end{proposition}
\medskip

\begin{definition}
A formally real field $\mathcal F$ is said to be {\it real-closed} if for any algebraic extension $\mathbb K/\mathcal F$ which is proper, $\mathbb K$ is not formally real.
    
\end{definition}
\medskip

\begin{proposition}[\cite{lam}, Chapter 8, Theorem 1.7]\label{real closed is Euclidean}
    If $\mathcal F$ is real-closed, then $\mathcal F$ is Euclidean.
\end{proposition}
\medskip

\begin{theorem}
\label{odd extension is formally real}
Let $\mathcal F$ be a formally real field. Let $\mathbb K$ be a field extension of $\mathcal F$ of an odd degree. Then $\mathbb K$ is also a formally real field.    
\end{theorem} 
\begin{proof}
Since $\mathcal F$ is formally real, we have $-1 \notin \sigma(\mathcal F)$. Consider an $n$-dimensional quadratic space $(V,q)$ over $\mathcal F$ with quadratic form $q=n \langle 1 \rangle= \underbrace {\langle 1,1, \dots, 1 \rangle}_{n\text{ times }}$. Then $q$ must be anisotropic over $\mathcal F$. Otherwise if $q$ is isotropic, we have 
  
\begin{align*}
     0&=x_1^2+x_2^2+\cdots +x_n^2, \text{ where } (x_1,x_2,\dots , x_n) \in V\\
     \text {i.e., }  -x_n^2&=x_1^2+x_2^2+\cdots +x_{n-1}^2
\end{align*}
    
    In particular, $-1$ is a sum of squares over $\mathcal F$, which is a contradiction. Therefore, the quadratic form $n \langle 1 \rangle$ becomes anisotropic over $\mathcal F$. 
    As $\mathbb K/\mathcal F$ is an odd degree extension, by Springer's theorem (see \cite{lam}, Chapter 7, Theorem 2.7), $ q \otimes_\mathcal F \mathbb K= n\langle 1 \rangle \otimes_\mathcal F \mathbb K $ is anisotropic over $\mathbb K$. Thus, $\mathbb K$ is formally real. 
\end{proof}
\begin{corollary} (\cite{lam}, Chapter 8, Corollary 2.3)
    Let $\mathcal F$ be a real-closed field. Then any odd degree polynomial in $\mathcal F[x]$ has a root in $\mathcal F$. 
\end{corollary}
\begin{proof}
    Since $\mathcal F$ is real-closed, it does not possess any algebraic extension that is formally real. Then by the above Theorem \ref{odd extension is formally real}, $\mathcal F$ can not have any proper odd degree extension.
 Consider a polynomial $f(x) \in \mathcal F[x]$ of odd degree $n$. Let $L$ be a splitting field of the polynomial $f(x)$ over $\mathcal F$. Since $[L:\mathcal F]=n$ is odd, and $\mathcal F$ has no non trivial odd-degree extension, we conclude that $L=\mathcal F$. Therefore, any odd degree polynomial in $ \mathcal F[x]$ possesses a root in $\mathcal F$.

\end{proof}
\section{Main theorem}
In this section, we discuss the proof of the Fundamental Theorem of Algebra, building upon the concepts and discussions from earlier sections. The idea of this proof is due to Artin and Schreier.
\medskip 

\begin{theorem}\label{Euclidean is quad closed}
    Let $\mathbb E$ be a Euclidean field. Then $\mathbb E(\sqrt{-1})$ is quadratically closed. 
\end{theorem}
\begin{proof}
    Consider an arbitrary non-zero element $a+(\sqrt{-1})b \in \mathbb E(\sqrt{-1})$. To prove that $\mathbb E(\sqrt{-1})$ is quadratically closed, it suffices to show that there exists $x,y \in \mathbb E$ such that  $a+(\sqrt{-1})b=(x+(\sqrt{-1})y)^2$. Expanding the square and comparing both sides, we get $a=x^2-y^2$ and $b=2xy$. Therefore $y=\frac{b}{2x}$. Substituting the value of $y$, we get
\begin{align*}
    & a=x^2- \frac{b^2}{4x^2}\\ 
     \text{i.e., } &x^4-ax^2-\frac{b^2}{4}=0
\end{align*}
    Solving, we get $x^2=\frac{1}{2}(a+\sqrt{a^2+b^2})$. As $\mathbb E$ is Pythagorean (by Proposition \ref{Euclidean is Pythagorean}), we have $a^2+b^2=c^2$ for some $c \in \mathbb E$. Thus, $x^2=\frac{1}{2}(a \pm c)$. 
    \medskip 
    
    \textbf{Case I:} If $a\geq c$, then $a \pm c \geq 0$. Since $E$ is Euclidean, we have $a \pm c =d^2$ for some $d \in \mathbb E$. Thus $x^2=\frac{1}{2}d^2$. Hence, $x= \pm \frac{1}{\sqrt{2}}d \in \mathbb E$ and $y=\frac{b}{2x} \in \mathbb E$.
\medskip 

    \textbf{Case II:} Suppose $a \pm c \leq 0$. Then we have $a \pm c = -d^2$ for some $d \in \mathbb E$. Thus  $x^2=-\frac{1}{2}d^2$. Hence, $x= \pm \frac{1}{\sqrt{2}}(\sqrt{-1})d \in \mathbb E(\sqrt{-1})$ and $y=\frac{b}{2x} \in \mathbb E(\sqrt{-1})$.
    
    Therefore, every element of $\mathbb E(\sqrt{-1})$ is a square. Thus, $\mathbb E(\sqrt{-1})$ is quadratically closed.
\end{proof}
\medskip 

\begin{theorem} \label{main theorem}
  Suppose $\mathcal F$ is a Euclidean field where every odd degree polynomial in $ \mathcal F[x]$ has a root within $\mathcal F$. Then $\mathcal F(\sqrt{-1})$ is algebraically closed. 
\end{theorem}
\begin{proof} 
Since $\mathcal F$ is a Euclidean field, by the above Theorem \ref{Euclidean is quad closed}, $\mathcal F(\sqrt{-1})$ is quadratically closed.
  Consider an algebraic extension $\mathbb K$ of $\mathcal F(\sqrt{-1})$ such that $[\mathbb K:\mathcal F(\sqrt{-1})]=2^\alpha \beta >1$, where $\beta$ is an odd integer.
    Let us consider the Galois closure $L$ (enough to consider a normal closure as ${\rm char}(\mathcal F)=0$) of $\mathbb K$ over $\mathcal F$.
    Let $H$ be a $2$-Sylow subgroup of the Galois group ${\rm Gal}(L/ \mathcal F)$ and let $L^ H$ denote its fixed field. Then we have, $[L:L^ H]=|H|=2^\alpha$. Since any  polynomial of odd degree in $ \mathcal F[x]$ has a root within $\mathcal F$ (by \cite{lam}, Chapter 8, Proposition 2.1), it follows that $\mathcal F$ does not admit extensions of odd degree. Hence, we have $L^ H = \mathcal F$, implying $[L^H:\mathcal F]=\beta = 1$. Therefore, $\mathcal F(\sqrt{-1})$ admits only $2$-extensions, i.e., any algebraic extension of $\mathcal F(\sqrt{-1})$ has degree $2^\alpha$ for some $\alpha$. On the other hand, since $\mathcal F(\sqrt{-1})$ is quadratically closed, by repeated application of Galois theory, we deduce that $L=\mathcal F(\sqrt{-1})$.
    Thus, any polynomial in $\mathcal F(\sqrt{-1})[x]$ has a root in $\mathcal F(\sqrt{-1})$, implying $\mathcal F(\sqrt{-1})$ is algebraically closed.
\end{proof}
\begin{corollary}
    The field of real numbers $\mathbb R$ is Euclidean, and $\mathbb R(\sqrt{-1})=\mathbb C$ is algebraically closed.
\end{corollary}
\begin{proof}
    As the field of real numbers $\mathbb R$ is a Pythagorean field with a unique ordering $R^{\times 2}$, by Proposition \ref{Euclidean is Pythagorean}, $\mathbb R$ is Euclidean. Consider an odd degree polynomial $f(x) \in \mathbb R[x]$. Any real-valued polynomial is also a continuous function over $\mathbb R$. 
    For an odd-degree polynomial, as $x$ approaches positive infinity, the polynomial tends toward positive or negative infinity depending on the leading coefficient's sign. Similarly, as $x$ approaches negative infinity, the polynomial tends toward the opposite sign. By the Intermediate Value Theorem for real numbers, there exists some $x_0 \in \mathbb R$ such that $f(x_0)=0$. Thus, any odd degree polynomial in $\mathbb R[x]$ has a root within $\mathbb R$. Hence, by Theorem \ref{main theorem}, $\mathbb R(\sqrt{-1}) =\mathbb C$ is algebraically closed.
\end{proof}
\medskip 

\noindent 
\textbf{Acknowledgements:} We would like to extend our sincere thanks to the anonymous referee for their insightful comments and constructive feedback, which significantly improved the quality of this manuscript. Their thorough review and valuable suggestions have been instrumental in refining our work. We deeply appreciate the time and effort devoted to evaluating our submission.

\end{document}